\documentclass[10pt,a4paper]{article}
\usepackage{amsmath}
\usepackage{amssymb}
\usepackage{amsthm}
\usepackage[pdftex]{color,graphicx}
\usepackage[utf8]{inputenc}
\usepackage[T1]{fontenc}
\usepackage[english]{babel}
\usepackage{bbm}
\usepackage[nice]{units}
\usepackage{hyperref}
\newtheorem{tw}{Theorem}[section]
\newtheorem{lm}[tw]{Lemma}
\newtheorem{wn}[tw]{Corollary}
\newtheorem{pr}[tw]{Proposition}
\theoremstyle{definition}
\newtheorem{df}{Definition}[section]
\newtheorem{uw}{Remark}[section]

\newcommand{\1}{\mathbbm{1}}
\newcommand{\xbm}{(X,{\cal B},\mu)}
\newcommand{\ycn}{(Y,{\cal C},\nu)}
\newcommand{\zdr}{(Z,{\cal D},\rho)}
\newcommand{\vep}{\varepsilon}
\newcommand{\Hinv}{H^{\otimes n}_{inv}(G)}
\title{Disjointness properties for Cartesian products of weakly mixing systems}

\author{Joanna Ku\l{}aga-Przymus\footnote{Research partially supported by MNiSzW grant N N201 384834 and Marie Curie "Transfer of Knowledge" program, project MTKD-CT-2005-030042 (TODEQ).}\\{\small Faculty of Mathematics and Computer Science,}\\ {\small Nicolaus Copernicus University},\\{\small ul. Chopina 12/18, 87-100 Toru\'n, Poland}\\ {\small e-mail: joanna.kulaga@gmail.com}\and Fran\c{c}ois Parreau\\{\small Laboratoire d'Analyse, G\'{e}om\'{e}trie et Applications,}\\{\small UMR 7539 Universit\'{e} Paris 13 et CNRS,}\\ {\small 99 av. J.-B. Cl\'{e}ment, 94430 Villetaneuse, France}\\ {\small e-mail: parreau@math.univ-paris13.fr}}

\begin{document}
\bibliographystyle{plain}
\maketitle

\begin{abstract}
For $n\geq 1$ we consider the class JP($n$) of dynamical systems whose every ergodic joining with a Cartesian product of $k$ weakly mixing automorphisms ($k\geq n$) can be represented as the independent extension of a joining of the system with only $n$ coordinate factors. For $n\geq 2$ we show that, whenever the maximal spectral type of a weakly mixing automorphism $T$ is singular with respect to the convolution of any $n$ continuous measures, i.e. $T$ has the so-called convolution singularity property of order $n$, then $T$ belongs to JP($n-1$). To provide examples of such automorphisms, we exploit spectral simplicity on symmetric Fock spaces. This also allows us to show that for any $n\geq 2$ the class JP($n$) is essentially larger than JP($n-1$). Moreover, we show that all members of JP($n$) are disjoint from ergodic automorphisms generated by infinitely divisible stationary processes.
\end{abstract}

\tableofcontents

\section{Introduction}
In this paper we deal with several properties of dynamical systems which are related to the notion of disjointness. This notion was introduced by Furstenberg~\cite{MR0213508} and, among other motivations, bore fruit in the development of tools to classify dynamical systems and construct examples of different behaviour. We will devote our attention to classes of automorphisms enjoying the so-called \emph{joining primeness property of order n} (JP($n$)) introduced in~\cite{MR2729082}. The JP($n$) class consists of automorphisms whose all ergodic joinings with Cartesian products of weakly mixing automorphisms are in fact joinings with at most $n$ coordinate factors, the remaining coordinate factors being joined by taking the product measure (for the precise definition see Section~\ref{sec:def}). In particular, weakly mixing automorphisms with the JP($n$) property do not admit a representation as Cartesian products of more than $n$ factors.

For each $n\geq 1$ we give examples of automorphisms enjoying the JP($n+1$) property and not the JP($n$) property (in~\cite{MR2729082} it was already shown that JP($1$)$\subsetneq$JP($2$)). It seemed natural to look for such examples among weakly mixing systems which are Cartesian products of $n+1$ copies of some automorphism, since such representation automatically implies that the JP($n$) property does not hold. The main tool which we use to find systems with the joining primeness property of a given order among Cartesian products is spectral theory, in particular the property of \emph{convolution singularity of order n} (CS($n$)). We say that a Borel measure $\mu$ on the circle group $\mathbb{T}$ has the CS($n$) property if it is singular with respect to the convolution of any $n$ continuous measures (see Section~\ref{sekcjaCS}). When $\mu$ is a convolution power of some measure, the CS($n$) property is tightly connected with the spectral simplicity of symmetric tensor products of some unitary operator. We will show that whenever the Gaussian system determined by the reduced maximal spectral type $\sigma_T$ of an automorphism $T$ has simple spectrum, then the $n$-th convolution power of $\sigma_T$ has the CS($n+1$) property. This, in turn, combined with the fact that all weakly mixing automorphisms having the CS($n+1$) property enjoy also the JP($n$) property will result in constructing examples announced in the beginning of this paragraph. As a byproduct, we give a proof of the folklore result that the spectrum of the infinite direct sum of unitary operators $\bigoplus_{n=1}^{\infty}U^{\odot n}$ is simple provided that the spectra of all of $U^{\odot n}$, $n\geq 1$ are simple.

The properties connected with spectral simplicity of symmetric tensor products were studied before by numerous authors. Let us list only some of them. In \cite{katokstepin1967} Katok and Stepin disproved Kolmogorov's conjecture (see eg.~\cite{sinai1963}) that for any ergodic system the convolution square of the maximal spectral type is absolutely continuous with respect to the maximal spectral type. Stepin in \cite{Stepin1968} showed that for a typical (with respect to the weak operator topology) non-mixing automorphism the convolution powers of the reduced maximal spectral type are pairwise disjoint: $\sigma_T^{\ast m}\perp \sigma_T^{\ast n}$ for $m\neq n$. Ageev in \cite{MR1680995} showed that even a stronger property is also typical for automorphisms: the Gaussian system generated by the reduced maximal spectral type has simple spectrum. Concrete examples of such systems include the Chacon automorphism (Ageev~\cite{MR1835446}), some mixing automorphisms (Ageev~\cite{MR2465595} and Ryzhikov~\cite{MR2354530}) and also recent examples in terms of special flows (Lema{\'n}czyk and Parreau~\cite{lemanczyk+parreau}).

In \cite{MR2729082} it was shown that automorphisms with the JP($1$) property are disjoint from dynamical systems arising from infinitely divisible (ID) stationary processes. There are many earlier results of a similar flavour. Let us mention here only one of them and refer the reader to the introduction of \cite{MR2729082} for a more exhaustive survey. The JP(1) class includes simple systems (introduced by Veech in~\cite{MR685378} and del Junco and Rudolph in \cite{MR922364}), which in turn contain the systems with the so-called minimal self-joining property (MSJ). In \cite{MR1325699} Thouvenot proved that systems with the MSJ property are disjoint from Gaussian systems. We extend the result from \cite{MR2729082} and show that all automorphisms satisfying the JP($n$) property for some $n\geq 1$ are disjoint from automorphisms arising from ID stationary processes. 

\section{Definitions}\label{sec:def}

\subsection{Tensor products}
Let $H$ be a separable Hilbert space. The space $F(H)=\bigoplus_{n=1}^{\infty}H^{\otimes n}$ is called a \emph{Fock space}. For a unitary operator $U \colon H \to H$, by $F(U)$ we denote the corresponding unitary operator acting on $F(H)$: $F(U)=\bigoplus_{n=1}^{\infty}U^{\otimes n}$. On each subspace $H^{\otimes n}\subset F(H)$ the action of the operator $U^{\otimes n}$ is determined by $U^{\otimes n}(x_1 \otimes \dots \otimes x_n)=Ux_1 \otimes\dots \otimes Ux_n$. 

Let $\pi\in S(n)$, where $S(n)$ stands for the permutation group of the set $\{1,\dots, n\}$. The operator $U_\pi \colon H^{\otimes n}\to H^{\otimes n}$ is defined by setting $U_{\pi} (x_{1} \otimes \dots \otimes x_n)= x_{\pi (1)} \otimes \dots \otimes x_{\pi(n)}$. By $H^{\odot n}$ we denote the space of all elements of $H^{\otimes n}$ invariant with respect to the action of $S(n)$, i.e. $H^{\odot n}=\{\widetilde{x}\in H^{\otimes n} \colon U_{\pi}(\widetilde{x})=\widetilde{x}\text{ for all }\pi \in S(n)\}$. By $U^{\odot n}$ we will understand $U^{\otimes n}|_{H^{\odot n}}$. The \emph{symmetric Fock space} is given by $F_{sym}(H)=\bigoplus_{n=1}^{\infty}H^{\odot n}$ and $F_{sym}(U)$ denotes the operator $\bigoplus_{n=1}^{\infty}U^{\odot n}=F(U)|_{F_{sym}(H)}$. More information about tensor products of Hilbert spaces and unitary operators can be found in~\cite{MR0272042}.

\subsection{Spectral theory}
For a unitary operator $U$ acting on a separable Hilbert space $H$ there exist elements $x_n\in H$ such that
\begin{equation}\label{rozklad}
H=\bigoplus_{n=1}^{\infty}\mathbb{Z}(x_n) \text{ and }\sigma_{x_1} \gg \sigma_{x_2} \gg \dots
\end{equation}
where $\mathbb{Z}(x)=\overline{span}\{U^n(x)\colon n\in \mathbb{Z}\}$ is the \emph{cyclic space} generated by the element $x\in H$ and $\sigma_x$ for $x\in H$ denotes the only finite, positive Borel measure on $\mathbb{T}=\{z\in\mathbb{C} \colon |z|=1\}$ such that $\langle U^nx,x\rangle=\int_{\mathbb{T}}z^n\ d\sigma_x(z)$ (this measure is called the \emph{spectral measure} of $x$). The set of all measures on $\mathbb{T}$ equivalent to $\sigma_{x_1}$ in the decomposition~\eqref{rozklad} is called the \emph{maximal spectral type} of $U$ and is denoted by $\sigma_U$. For $n\geq 2$ the maximal spectral type of the tensor product of operators $U_1,\dots,U_n$ is given by the convolution of their maximal spectral types: $\sigma_{U_1\otimes\dots\otimes U_n}=\sigma_{U_1} \ast\dots\ast \sigma_{U_n}$.

For $n\geq 1$ let $A_n=\{z\in\mathbb{T}\colon \frac{d\sigma_{x_n}}{d\sigma_{x_1}}\neq 0 \}$, where $\frac{d\sigma_{x_n}}{d\sigma_{x_1}}$ is the Radon-Nikodym derivative. The \emph{spectral multiplicity function} $M_{U}\colon \mathbb{T}\to\mathbb{N} \cup \{\infty\}$ is given by $M_U(z)=\sum_{n=1}^{\infty}\1_{A_n}(z)$ and the \emph{spectral multiplicity} of $U$ is the essential supremum of $M_{U}$. If for some $N\geq 1$ $M_U=N$ almost everywhere with respect to the maximal spectral type, we say that $U$ has \emph{homogeneous spectrum of multiplicity $N$}. If moreover $N=1$, we say that $U$ has \emph{simple spectrum}.

Consider the operator $V_{\sigma}\colon L^2(\mathbb{T},\sigma) \to L^2(\mathbb{T},\sigma)$ given by $V_{\sigma}(f)(z)=z\cdot f(z)$, where $\sigma$ is a finite, positive Borel measure on $\mathbb{T}$. It has simple spectrum and its maximal spectral type is $\sigma$. The operator $U|_{\mathbb{Z}(x)} \colon \mathbb{Z}(x)\to\mathbb{Z}(x)$ is spectrally equivalent to $V_x\colon L^2(\mathbb{T},\sigma_x)\to L^2(\mathbb{T},\sigma_x)$. For an introduction to the spectral theory of unitary operators see e.g.~\cite{MR2140546} and for other locally compact abelian group actions see~\cite{MR2186251} or~\cite{lem-ency}.

Given an automorphism $T\colon (X,\mathcal{B},\mu)\to (X,\mathcal{B},\mu)$ of a standard probability Borel space, by its spectral properties we will understand the spectral properties of the associated unitary operator (the so-called \emph{Koopman operator}) $U_T\colon L^2(X,\mathcal{B},\mu)\to L^2(X,\mathcal{B},\mu)$ defined by $U_T(f)=f\circ T$, e.g. the maximal spectral type of $T$ is the maximal spectral type of $U_T$. Since any Koopman operator has an atom at $1$, i.e. $\delta_1\ll \sigma_{U_T}$, we will also use the notion of the reduced maximal spectral type, i.e. $\sigma_{U_T|_{L^2_0(X)}}$, where $L^2_0(X)=\{f\in L^2(X)\colon \int_X f\ d\mu=0\}$.

We recall that spectral simplicity of $F_{sym}(U_T|_{L^2_0(X)})$ is equivalent to spectral simplicity of the Gaussian automorphism associated to the reduced maximal spectral type of $U_T$\footnote{For more information concerning Gaussian systems we refer the reader e.g. to~\cite{MR0272042}.}.

\subsection{Joinings}
Let $T$ and $S$ be automorphisms of standard probability Borel spaces $(X,\mathcal{B},\mu)$ and $(Y,\mathcal{C},\nu)$ respectively. By $J(T,S)$ we denote the set of all \emph{joinings} between $T$ and $S$, i.e. the set of all $T\times S$-invariant probability measures on $(X\times Y,\mathcal{B}\otimes \mathcal{C})$, whose projections on $X$ and $Y$ are equal to $\mu$ and $\nu$ respectively. The subset of $J(T,S)$ consisting of ergodic joinings will be denoted by $J^e(T,S)$. For $J(T,T)$ we write $J(T)$. Joinings are in one-to-one correspondence with Markov operators $\Phi\colon L^2(X,\mathcal{B},\mu)\to L^2(Y,\mathcal{C},\nu)$ satisfying $\Phi\circ T=S \circ \Phi$: 
\begin{align*}
&\Phi \mapsto \lambda_{\Phi} \in J(T,S),\ \lambda_{\Phi}(A\times B)= \int_B \Phi(\1_A)\ d\nu,\\
&\lambda \mapsto \Phi_{\lambda}, \int \Phi_{\lambda}(f)(y)g(y)\ d\nu(y)=\int f(x)g(y)\ d\lambda(x,y).
\end{align*}
We denote the set of such Markov operators by $\mathcal{J}(T,S)$ and endow it with the weak operator topology. This identification allows us to view $J(T)$ as a metrizable compact semitopological semigroup. If $\mathcal{A}\subset \mathcal{C}$ is a factor of $S$ and $\lambda\in J(T,S|_{\mathcal{A}})$ we can consider the joining $\hat{\lambda}\in J(T,S)$ given on $\mathcal{B}\otimes \mathcal{C}$ by
$$
\hat{\lambda}(B\times C)=\int_{X\times Y}\1_{B}(x)E(\1_{C}|\mathcal{A})(y)\ d\lambda(x,y).
$$
It is called the \emph{relatively independent extension} of $\lambda$.

We say that $T$ and $S$ are disjoint if $J(T,S)=\{\mu\otimes\nu\}$ (the notion of disjointness was introduced by H.~Furstenberg in~\cite{MR0213508}). We refer the reader to~\cite{MR1958753} for more information on the theory of joinings and e.g. to~\cite{MR1784644} for a short survey on the basic notions.

\subsection{Convolution singularity}\label{sekcjaCS}
Throughout the paper, by measure on $\mathbb{T}$ we will always mean a positive finite Borel measure.
\begin{df}[\cite{MR2729082}]
A measure $\sigma$ on $\mathbb{T}$ has the \emph{convolution singularity property} (CS) if $\sigma \perp \nu_1\ast \nu_2$ for any continuous measures $\nu_1,\nu_2$ on $\mathbb{T}$. An automorphism $T$ has the CS property if its maximal spectral type has this property.
\end{df}
We can generalize this property and consider singularity with respect to convolutions of more than two measures.
\begin{df}
A measure $\sigma$ on $\mathbb{T}$ has the \emph{convolution singularity property of order $n$} (CS($n$)) if $\sigma \perp \nu_1\ast \dots \ast\nu_n$ for any continuous measures $\nu_1$, \dots, $\nu_n$ on $\mathbb{T}$. An automorphism $T$ has the CS($n$) property if its maximal spectral type has this property.
\end{df}

\subsection{Joining primeness}
The notion of joining primeness (JP) and its generalization - joining primeness of higher order - were introduced in~\cite{MR2729082}.
\begin{df}
Let $T\colon \xbm\to \xbm$ be an ergodic  automorphism of a standard probability Borel space. We say that $T$ has the \emph{joining primeness property of order} $n\geq1$
(JP($n$)) if for any $k\geq n+1$ and any weakly mixing automorphisms $S_i\colon (Y_i,\mathcal{C}_i,\nu_i)\to (Y_i,\mathcal{C}_i,\nu_i)$, $1\leq i\leq k$, for every $\lambda\in J^{e}\left(T,S_{1}\times\ldots\times S_{k}\right)$ there exist $i_{1},\dots,i_{n}$, $1\leq i_{j}\leq k$ such that
\begin{equation*}
\lambda=\lambda_{X,Y_{i_1},\dots,Y_{i_n}}\otimes \otimes_{j\not\in \{i_1,\dots,i_n\}}\nu_j,
\end{equation*}
where $\lambda_{X,Y_{i_1},\dots,Y_{i_n}}$ is the projection of $\lambda$ on the product of the corresponding coordinate factors.
\end{df}

Speaking less formally, the JP($n$) property means that ergodic joinings with Cartesian products of weakly mixing automorphisms are in fact independent extensions of joinings with products of at most $n$ factors.

\begin{uw}(\cite{MR2729082})
Adding the requirement that $S_1,\dots,S_k$ are isomorphic yields an equivalent notion. Also restricting the definition by fixing $k=n+1$ (instead of taking an arbitrary $k\geq n+1$) brings nothing new.
\end{uw}

\begin{uw}
Clearly JP($m$)$\subset$JP($n$) for $m\leq n$.
\end{uw}

\section{General results}
\subsection{CS($n$) property vs spectral multiplicities of tensor products on some subspaces}\label{sekcja3}
Our results from this section overlap with the results from~\cite{MR2604539} (e.g.~Proposition~\ref{krot} follows from the ``calculus of spectral multiplicities'' proposed in~\cite{MR2604539}). However, for the sake of completeness and for the consistency of the language we have decided to include the overlaping results as well. Ryzhikov also provides in~\cite{MR2604539} new examples of spectral multiplicities, including e.g. $\{p, q, pq\}$, $\{p, q, r, pq, pr, rq, pqr\}$ (cf. Remark~\ref{robinson}). Recall that the problem to determine which of the subsets of $\mathbb{N}$ can be obtained as the sets of essential values of the spectral multiplicity function remains open.

For $n\geq 1$, we denote by $C_n$ the map from $\mathbb{T}^{n}$ to $\mathbb{T}$ given by
\begin{equation*}
C_n(z_1,\dots,z_n)=z_1 \cdot \ldots \cdot z_n.
\end{equation*}
So, if $U$ is a unitary operator of a separable Hilbert space, $\sigma_{U^{\otimes n}}=(\sigma_U)^{\ast n}=(C_n)_{\ast}(\sigma_U^{\otimes n})$. Throughout the paper we will use the following well-known characterization.

\begin{pr}\label{tw42}
Let $\sigma$ be a finite positive Borel measure on $\mathbb{T}$. The operator $V_{\sigma}^{\otimes n}$ has homogeneous spectrum of multiplicity $N$ if and only if the map $C_n\colon \mathbb{T}^n\to\mathbb{T}$ is $N$-to-one on some Borel set $F\subset \mathbb{T}^n$ with $\sigma^{\otimes n}(F^c)=0$.
\end{pr}

\begin{uw}
The same property can be also expressed in terms of disintegration of measures into conditional measures, as it was done in~\cite{MR2186251} in the case of two operators with simple spectrum (whose product has not necessarily a homogeneous spectrum):
\newline
\emph{
The operator $V_\sigma^{\otimes n}$ has homogeneous spectrum of multiplicity $N$ if and only if in the disintegration
\begin{equation*}
\sigma^{\otimes n}=\int_{\mathbb{T}} \mu_z\ d \sigma^{\ast n}(z)
\end{equation*}
for $\sigma^{\ast n}$-almost every $z$, the measure $\mu_z$ (which is concentrated on $C_n^{-1}(z)$) is purely atomic and has $N$ atoms.
}
\end{uw}

\begin{uw} \label{uproste}
For a unitary operator $U \colon H \to H$ the condition that $U^{\odot k}$ has simple spectrum for some $k\geq 1$ implies that $U$ has simple spectrum too. Indeed, otherwise we can find $x_1, x_2 \in H$ with same spectral measures $\sigma_{x_1}= \sigma_{x_2}$ and such that $\mathbb{Z}(x_1) \perp \mathbb{Z}(x_2)$. Then $x_1^{\otimes k}, x_2^{\otimes k} \in H^{\odot k}$, $\mathbb{Z}(x_1^{\otimes k}) \perp \mathbb{Z}(x_2^{\otimes k})$ and $\sigma_{x_1^{\otimes k}}= \sigma_{x_2^{\otimes k}}$, which implies that spectrum of $U^{\odot k}$ is not simple either.
\end{uw}

Let $I=\{ \vec{i}=(i_{1},\ldots,i_{n})\colon \{i_1,\dots,i_n\}=\{1,2,\dots,n\} \}$. Let $G$ be a subgroup of the group $S(n)$, acting naturally on $I$ by
\begin{equation*}
\pi((i_{1},\dots,i_{n}))=(\pi(i_{1}),\ldots,\pi(i_{n}))
\end{equation*}
for $\pi \in G$. 
Denote by $o_G$ the number of orbits of this action of $G$ on $I$. Since every orbit has $\# G$ elements, we have $o_G=\frac{n!}{\# G}$. Now we consider the space 
$$
H^{\otimes n}_{inv}(G)=\left\{\widetilde{x}\in H^{\otimes n}\colon U_{\pi}(\widetilde{x})=\widetilde{x} \text{ for }\pi \in G   \right\}.
$$
\begin{uw}
Let $n\geq 1$. If $\sigma=\sigma_x$ is the maximal spectral type of the unitary operator $U\colon H\to H$ then the spectral measure of $(x,\dots,x)$ under $U^{\otimes n}$ is  $\sigma_x^{\ast n}$, so it is the maximal spectral type of $U^{\otimes n}$. Since $(x,\dots, x)\in H^{\otimes n}_{inv}(G)$ for any subgroup $G\subset S(n)$, the maximal spectral type of  $U^{\otimes n}|_{H^{\otimes n}_{inv}(G)}$ is also $\sigma^{\ast n}$. In particular $\sigma_{U^{\odot n}}=\sigma^{\ast n}$.

\end{uw}
\begin{uw}\label{rzut}
Notice that the orthogonal projection on $H^{\otimes n}_{inv}(G)$ is given by
\begin{equation*}
proj_{H^{\otimes n}_{inv}(G)}=\frac{1}{\# G}\sum_{\pi \in G}U_{\pi}.
\end{equation*}
\end{uw}

The spectral decompositions of the operators $U^{\odot n}$ and $U^{\otimes n}|_{H^{\otimes n}_{inv}(G)}$ are connected with each other. In case when the spectrum of $U^{\odot n}$ is simple and the maximal spectral type is continuous, we obtain the following characterization.
\begin{pr}\label{sigmaperm}
Let $U\colon H\to H$ be a unitary operator of a separable Hilbert space. Assume that $\sigma=\sigma_U$ is a continuous measure and that $U^{\odot n}$ has simple spectrum. Then $U^{\otimes n}|_{\Hinv}$ has homogeneous spectrum of multiplicity $o_G$.
\end{pr}
A similar theorem can be found in Ageev~\cite{MR2465595}. Even though our situation is simpler (in~\cite{MR2465595} the measure has one atom, whereas here it is continuous), we include the proof for the sake of completeness.

\begin{proof}
Since by Remark~\ref{uproste} the operator $U$ has simple spectrum, we can assume that $U=V_{\sigma_U}$, that is
\begin{equation*}
H=L^2(\mathbb{T},\sigma_U) \textrm{ and } Uf(z)=zf(z).
\end{equation*}
We will use the ordering on $\mathbb{T}$ inherited from principal values of arguments of complex numbers. Let
\begin{equation*}
A_{\vec{i}}=\left\{(z_1,\ldots,z_n)\in \mathbb{T}^n;\ z_{i_1}<\ldots < z_{i_n}\right\} \text{ and }F_{\vec{i}}=\1_{A_{\vec{i}}}L^2\left(\mathbb{T}^n,\sigma_U^{\otimes n}\right)
\end{equation*}
for $\vec{i} \in I$.
The sets $A_{\vec{i}}$ are pairwise disjoint and we have
\begin{equation}\label{eq:pepepe}
F_{\vec{i}}\perp F_{\vec{j}}\text{ for }\vec{i} \neq \vec{j}.
\end{equation}
By continuity of $\sigma$, $\mathbb{T}^n$ can be decomposed up to a set of measure zero into the disjoint union $\bigcup_{\vec{i}\in I}A_{\vec{i}}$. Therefore $H^{\otimes n}=\bigoplus_{\vec{i}\in I} F_{\vec{i}}.$
We obtain
\begin{equation*}
\Hinv=proj_{\Hinv}H^{\otimes n}=proj_{\Hinv}\bigoplus_{\vec{i}\in I} F_{\vec{i}}.
\end{equation*}
Notice that
\begin{equation}\label{pr}
U_{\pi}(F_{\vec{i}}) = F_{\pi^{-1}(\vec{i})} \text{ for }\pi \in S(n) \text{ and }\vec{i} \in I.
\end{equation}
If $\vec{j}=\tau(\vec{i})$ for some $\tau \in G$, by Remark~\ref{rzut}
\begin{equation*}
proj_{\Hinv} F_{\vec{i}} = proj_{\Hinv} F_{\vec{j}}.
\end{equation*}
If $\vec{j}$ is not in the $G$-orbit of $\vec{i}$, then for any $\pi,\ \pi ' \in G$ we have $\pi (\vec{i}) \neq \pi ' (\vec{j}),$ and hence the sets $A_{\pi^{-1}(\vec{i})} \textrm{ and } A_{\pi'^{-1} (\vec{j})}$ are disjoint and $F_{\pi^{-1} (\vec{i})} \perp F_{\pi'^{-1} (\vec{j})}.$ It follows by Remark~\ref{rzut} that
\begin{equation*}
proj_{\Hinv} F_{\vec{i}} \perp proj_{\Hinv} F_{\vec{j}}.
\end{equation*}
This means that for some $\vec{i}_1,\dots, \vec{i}_{o_G}$ we have
\begin{equation*}
\Hinv=\bigoplus_{k=1}^{o_G}proj_{\Hinv}F_{\vec{i_k}}.
\end{equation*}

Now,
\begin{equation*}
U^{\otimes n}\vert _{F_{\vec{i}}} \simeq U^{\otimes n}\vert _{proj_{H_{inv}^{\otimes n}(G)}{F_{\vec{i}}}}.
\end{equation*}
Indeed, the unitary equivalence is given by
\begin{equation*}
\1_{A_{\vec{i}}}f \mapsto \sqrt{\# G}\ proj_{H_{inv}^{\otimes n}(G)}f.
\end{equation*}
The isometricity condition follows by Remark \ref{rzut}, by \eqref{pr},~\eqref{eq:pepepe} and by the fact that $U_{\pi}$ is an isometry of  $L^2(\mathbb{T}^n,\sigma_{U}^{\otimes n})$. Moreover, $proj_{H^{\otimes n}_{inv}(G)} \circ U^{\otimes n} = U^{\otimes n} \circ proj_{H^{\otimes n}_{inv}(G)}$\footnote{We have $proj_{F}\circ W = W \circ proj_F$ for any bounded linear operator $W$ on a Hilbert space $H$ and any $W$- and $W^\ast$-invariant subspace $F\subset H$.}.
Hence
\begin{equation*}
U^{\otimes n}|_{\Hinv}\simeq U^{\otimes n}|_{\bigoplus_{k=1}^{o_G}F_{\vec{i_k}}}.
\end{equation*}
Moreover, for any $\vec{i},\vec{j} \in I$ it holds $U^{\otimes n}\vert _{F_{\vec{i}}} \simeq U^{\otimes n}\vert _{F_{\vec{j}}}$ (the isomorphism is given by $U_{\pi}$ for the appropriate $\pi \in S(n)$).
Therefore
\begin{equation}\label{333}
U^{\otimes n}|_{\Hinv}\simeq \bigoplus_{k=1}^{o_G} U^{\otimes n}|_{F_{(1,\ldots,n)}}.
\end{equation}
To complete the proof, it suffices to show that $U^{\otimes n}|_{F_{(1,\dots,n)}}$ has simple spectrum. This is however true since the condition~\eqref{333} for $G=S(n)$ means that
\begin{equation}\label{334}
U^{\odot n}\simeq U^{\otimes n}|_{F_{(1,\dots,n)}}.
\end{equation}
\end{proof}

\begin{wn}
Let $n\geq 1$. If the maximal spectral type of $U$ is continuous then $U^{\otimes n}$ has homogeneous spectrum of multiplicity $n!$ if and only if $U^{\odot n}$ has simple spectrum.
\end{wn}
\begin{proof}
Fix $n\geq 1$. If the operator $U^{\odot n}$ has simple spectrum then applying Proposition~\ref{sigmaperm} to $G=\{Id\}$ we see that $U^{\otimes n}$ has homogeneous spectrum of multiplicity $n!$

Now assume that $U^{\odot n}$ does not have simple spectrum. First we consider the case when $U$ itself has simple spectrum. By conditions~\eqref{334} and~\eqref{333} from Proposition~\ref{sigmaperm} we have
$$
U^{\otimes n}\simeq \bigoplus_{k=1}^{n!} U^{\otimes n}|_{F_{(1,\dots,n)}}\simeq \bigoplus_{k=1}^{n!}U^{\odot n}.
$$
Therefore the spectral multiplicity of $U^{\otimes n}$ is at least equal to $2n!$. Now if $U$ does not have simple spectrum then $H \supset \mathbb{Z}(x_1)\oplus \mathbb{Z}(x_2)$ for some $x_1,x_2\in H$ with $\sigma_{x_1}= \sigma_{x_2}$. Therefore $H^{\otimes n}\supset \mathbb{Z}(x_1)^{\otimes n}\oplus \mathbb{Z}(x_2)^{\otimes n}$ and $H^{\odot n}\supset \mathbb{Z}(x_1)^{\odot n}\oplus \mathbb{Z}(x_2)^{\odot n}$. Using the same arguments as before, the spectral multiplicity is thus again at least equal to $2n!$
\end{proof}

\begin{uw}\label{uw:nowa}
Let $n\geq 1$ and assume that $U$ has simple spectrum and its maximal spectral type is continuous. As in the above corollary, it follows from conditions~\eqref{334} and~\eqref{333} in Proposition~\ref{sigmaperm} that 
\begin{equation}\label{ryz}
U^{\otimes n}\simeq  \bigoplus_{k=1}^{n!}U^{\odot n}.
\end{equation}
Therefore $U^{\odot n}$ has homogeneous spectrum of multiplicity $k$ for some $k\geq 1$ if and only if $U^{\otimes k}$ has homogeneous spectrum of multiplicity $n!\cdot k$. As noted in~\cite{MR2604539}, formula~\eqref{ryz} remains true also in the general case, i.e. without assuming simplicity of spectrum of $U$.
\end{uw}

\begin{lm} \label{vproste}
If for some $k\geq 1$ the operator $U^{\odot k}$ has simple spectrum then also for $1\leq j\leq k-1$ the operators $U^{\odot j}$ have simple spectrum.
\end{lm}
\begin{proof}
By Remark~\ref{uproste}, we may assume that $U=V_\sigma$ where $\sigma=\sigma_U$, and we may moreover assume that $\sigma$ is a probability measure. We will show that, when $k\ge 2$, spectral simplicity of $V_{\sigma}^{\odot k}$ implies spectral simplicity for $V_{\sigma}^{\odot (k-1)}$.
Suppose that $V_{\sigma}^{\odot (k-1)}$ has not simple spectrum. This means that
\begin{equation}\label{ghgh}
\begin{array}{l}
\mbox{for any Borel $F \subset \mathbb{T}^{k-1}$, $\sigma^{\otimes (k-1)}(F)=1$
the map $C_{k-1}$}\\
\mbox{is not one-to-one modulo coordinate permutations on $F$.}
\end{array}
\end{equation}
Take any Borel set $E \subset \mathbb{T}^k$ such that $\sigma^{\otimes k}(E)=1$. We claim that $C_k$ is not one-to-one modulo coordinate permutations on $E$.
Let $F\subset \mathbb{T}^{k-1}$ be defined by
\begin{equation*}
F=\left\{(z_1,\dots,z_{k-1})\in \mathbb{T}^{k-1}\colon \sigma \left(\left\{ y\in\mathbb{T}\colon (z_1,\dots,z_{k-1},y)\in {E} \right\}\right)=1 \right\}.
\end{equation*}
Then $\sigma^{\otimes (k-1)}(F)=1$. By~\eqref{ghgh} it follows that there exist $(z_1,\dots,z_{k-1})$ and $(z'_{1},\dots,z'_{k-1})$ in $F$ with $C_{k-1}(z_1,\dots,z_{k-1})=C_{k-1}(z'_1,\dots,z'_{k-1})$ which cannot be obtained from one another by a coordinate permutation.

By the definition of $F$, we can find $z\in \mathbb{T}$ such that both $(z_1,\dots,z_{k-1},z)$ and $(z'_1,\dots,z'_{k-1},z)$ are in ${E}$. Clearly $C_k(z_1,\dots,z_{k-1},z)=C_k(z'_1,\dots,z'_{k-1},z)$ and these points are not either equal modulo coordinate permutations, which completes the proof.
\end{proof}

\begin{pr}\label{krot}
Let $k,m\geq 1$. Assume that $\sigma$ is a continuous measure on $\mathbb{T}$ and that the operator $V_{\sigma}^{\odot mk}$ has simple spectrum. Then the operator $(V_{\sigma^{\ast k}})^{\otimes m}$ has a homogeneous spectrum of multiplicity $\displaystyle{\frac{(mk)!}{(k!)^m}}$ with maximal spectral type $\sigma^{\ast mk}$.
\end{pr}

\begin{proof}
The maximal spectral type of $V_{\sigma^{\ast k}}^{\otimes m}$ is $(\sigma^{\ast k})^{\ast m}=\sigma^{\ast mk}$. We have to determine its spectral multiplicity.

Since by Lemma~\ref{vproste} the operator $V_{\sigma}^{\odot k}$ has simple spectrum and its maximal spectral type is $\sigma^{\ast k}$, we have $V_{\sigma^{\ast k}}\simeq V_\sigma^{\odot k}$. Hence
$$
(V_{\sigma^{\ast k}})^{\otimes m} \simeq (V_\sigma^{\odot k})^{\otimes m} = \left( V_\sigma^{\otimes k}|_{H^{\odot k}}\right)^{\otimes m}=V_\sigma^{\otimes km}|_{(H^{\odot k})^{\otimes m}},
$$
where $H=L^2(\mathbb{T},\sigma)$.

We identify $\{1,\dots,mk\}$ with $\{1,\dots,k\}\times \{1,\dots ,m\}$. Consider the subgroup $G$ of $S(mk)$ of all permutations
$$
\pi \colon  (i,j)\mapsto (\pi_j(i),j),\ \ 1\leq i\leq k, 1\leq j\leq m,
$$
where $\pi_j\in S(k)$ for $1\leq j\leq m$. Then $\# G=(k!)^m$ and it follows from Remark~\ref{rzut} that
$$
proj_{H_{inv}^{\otimes mk}(G)}=\frac{1}{(k!)^m}\sum_{\pi\in G}U_{\pi}=\frac{1}{(k!)^m}\sum_{(\pi_1,\ldots,\pi_m)\in S(k)^m}U_{\pi_1}\otimes\cdots\otimes U_{\pi_m}=\otimes_{j=1}^{m}proj_{H^{\odot k}},
$$
so
$$
(H^{\odot k})^{\otimes m}=H_{inv}^{\otimes mk}(G).
$$
Therefore, as $V_\sigma^{\odot mk}$ has simple spectrum, by Proposition~\ref{sigmaperm}, ${V_\sigma^{\otimes mk}}|_{({H^{\odot k}})^{\otimes m}}$ has homogeneous spectrum of multiplicity
$$
o_G=\frac{(mk)!}{\# G}=\frac{(mk)!}{(k!)^m}.
$$

\end{proof}
Using Proposition~\ref{krot} and Remark~\ref{uw:nowa}, we obtain the following.
\begin{wn}\label{wn:bububu}
Let $\sigma$ be a continuous measure on $\mathbb{T}$ such that the spectrum of $F_{sym}(V_{\sigma})$ is simple. Then for $k\geq 1$ and $m\geq 1$ the operator $\left(V_{\sigma^{\ast k}} \right)^{\odot m}$ has homogeneous spectrum of multiplicity $\frac{(mk)!}{(k!)^m} \cdot \frac{1}{m!}$.
\end{wn}

\begin{uw}\label{robinson}
Notice that Corollary~\ref{wn:bububu} yields a generalization of the example provided by A.~I.~Danilenko and V.~V.~Ryzhikov in~\cite{MR2794951} which shows that there exists a unitary operator $U$ whose set of spectral multiplicities of $\oplus_{m\geq 1}U^{\odot m}$ is equal to 
\begin{equation*}
\{1,1\cdot 3,1\cdot 3\cdot 5,1\cdot 3\cdot 5\cdot 7,\ldots\}.
\end{equation*}
This is the special case of Corollary~\ref{wn:bububu} with $k=2$ and it shows that the set of spectral multiplicities for a Gaussian system need not be a multiplicative sub-semigroup of $\mathbb{N}$ (contrary to the claim of A.~E.~Robinson from~\cite{robi}).
\end{uw}

\begin{uw}\label{hip}
Let $\mu_1,\dots,\mu_n$ be continuous measures on $\mathbb{T}$. Then by Fubini's theorem
\begin{equation*}
\mu_1 \otimes \dots \otimes \mu_n ( \{(z_1,\dots,z_n) \in \mathbb{T}^n;\ z_1^{\varepsilon_1}\cdot \ldots \cdot z_n^{\varepsilon_n}=c \})=0,
\end{equation*}
for $\varepsilon_1,\dots,\vep_n \in \{-1,0,1\}$ with $\sum_{i=1}^n \varepsilon_i^2 \neq 0$ and every $c \in \mathbb{T}$.
\end{uw}

\begin{tw}\label{gltw}
Let $\sigma$ be a continuous measure on $\mathbb{T}$. If $V^{\odot mk}_{\sigma}$ has simple spectrum, then $\sigma^{\ast k}$ has the $CS($n$)$ property for any $n$ such that such that $(m!)^n>(mk)!/(k!)^m$.
\end{tw}
This theorem is proved for the case where $k=1$ and $m=2$ in~\cite{pa-ro}. Here we provide the proof in the whole generality.
\begin{proof}
Let $k,n,m\in\mathbb{N}$ be such that $(m!)^n>\frac{(mk)!}{(k!)^m}$. Assume that for some continuous measures $\sigma_1, \ldots , \sigma_n$ we have
\begin{equation}\label{nieort}
\sigma^{\ast k} \not\perp \sigma_1 \ast \ldots \ast \sigma_n.
\end{equation}
We may assume that $\sigma,\sigma_1,\dots,\sigma_n$ are probability measures. By Propositions~\ref{krot} and \ref{tw42}, the map $C_m$ is $\frac{(mk)!}{(k!)^m}$-to-one on some set $F\subset \mathbb{T}^{m}$ with $(\sigma^{\ast k})^{\otimes m}(F)=1$.
We claim that under our assumptions this yields a contradiction:  we will find $(m!)^n$ distinct points from the set $F$ with the same product of coordinates. It is not possible since by assumption $(m!)^n>\frac{(mk)!}{(k!)^m}$.

Since we have assumed \eqref{nieort}, there exists a Borel set $A\subset \mathbb{T}$ with $\sigma_1 \ast \ldots \ast \sigma_n (A) >0$ such that
\begin{equation*}
\sigma_1 \ast \ldots \ast \sigma_n \vert _{A} \ll \sigma ^{\ast k}.
\end{equation*}
Therefore
\begin{equation*}
(\sigma_1 \ast \ldots \ast \sigma_n)^{\otimes m} (A^{m}\cap F^c )=0,
\end{equation*}
whence
\begin{equation}\label{anzero}
(\sigma_1 \otimes \ldots \otimes \sigma_n)^{\otimes m}(\widetilde{A}^{m} \cap \widetilde{F}^c)=0,
\end{equation}
where $\widetilde{A}=C_n^{-1}(A)$ and $\widetilde{F}=((C_n)^{m})^{-1}(F)$.

Now fix $\varepsilon >0$. There exist sets $B_1, \dots, B_n \in \mathcal{B}(\mathbb{T})$ such that $\sigma_1 \otimes \ldots \otimes \sigma_n (B_1 \times \ldots \times B_n)>0$ and the ``parallelepiped'' $B_1 \times \ldots \times B_n$ is included in~$\widetilde{A}$ up to~$\varepsilon$, precisely speaking
\begin{equation}\label{male}
\sigma_1 \otimes \ldots \otimes \sigma_n (B_1 \times \ldots \times B_n \setminus \widetilde{A}) < \varepsilon\ \sigma_1 \otimes \ldots \otimes \sigma_n (B_1 \times \ldots \times B_n).
\end{equation}

We identify again $\{1,\dots,mn\}$ with $\{1,\dots,n\}\times \{1,\dots ,m\}$.
Let now $G \subset S(mn)$ be the subgroup of permutations of the form $\pi=(\pi_i)_{1\leq i \leq n}$ with $\pi_i \in S(m)$ ($1\le i\le n$), defined by
\begin{equation*}
\pi \colon  (i,j)\mapsto (i,\pi_i(j)),\ 1\leq i\leq n, 1\leq j\leq m,
\end{equation*}
and acting on $(\mathbb{T}^{n})^{m}$ by
\begin{equation*}
(z_{i,j})_{1\leq i\leq n,\,1\leq j \leq m}\mapsto (z_{\pi(i,j)})_{1\leq i\leq n,\,1\leq j \leq m}.
\end{equation*}
Notice that such permutations $\pi$ preserve the measure $(\sigma_1 \otimes \ldots \otimes \sigma_n)^{\otimes m}$ and that the sets $(B_1\times \ldots \times B_n)^{m}$ are invariant under the action of $G$. Without loss of generality we can also assume that $\widetilde{F}$ is invariant under $G$ (we can restrict to $\bigcap_{\pi\in G}\pi^{-1}(\widetilde{F})$, which is still of full measure).

\vskip 1ex
Since
\begin{equation*}
(B_1\times \dots \times B_n)^{m} \setminus \widetilde{A}^{ m}
\subset\bigcup_{j=1}^m (\mathbb{T}^{n})^{j-1}\times
\left( B_1\times \dots \times B_n \setminus \widetilde{A}\right)
\times (\mathbb{T}^{n})^{m-j},
\end{equation*}
we have by~\eqref{male}
\begin{multline*}
(\sigma_1 \otimes \ldots \otimes \sigma_n)^{\otimes m} \left( (B_1 \times \ldots \times B_n)^{m}\setminus \widetilde{A}^{m} \right)\leq
\\
\leq m\ \varepsilon\cdot \sigma_1 \otimes \ldots \otimes \sigma_n(B_1 \times \ldots \times B_n) ,
\end{multline*}
whence for all $\pi \in G$
\begin{multline*}
(\sigma_1 \otimes \ldots \otimes \sigma_n)^{\otimes m} \left( (B_1 \times \ldots \times B_n)^{m}\setminus \pi^{-1}(\widetilde{A}^{m}) \right)\leq
\\
\leq m\ \varepsilon\cdot \sigma_1 \otimes \ldots \otimes \sigma_n(B_1 \times \ldots \times B_n) .
\end{multline*}
Therefore
\begin{multline*}
(\sigma_1 \otimes \ldots \otimes \sigma_n)^{\otimes m} \left( (B_1 \times \ldots \times B_n)^{m}\setminus \bigcap_{\pi \in G} \pi^{-1}(\widetilde{A}^{m}) \right)\leq
\\
\leq \# G \cdot m\varepsilon\cdot \sigma_1 \otimes \ldots \otimes \sigma_n(B_1 \times \ldots \times B_n).
\end{multline*}
So, if $\varepsilon$ is small enough, by~\eqref{anzero}
\begin{equation*}
(\sigma_1 \otimes \ldots \otimes \sigma_n)^{\otimes m}\left(  (B_1 \times \ldots \times B_n)^{m}\cap \bigcap_{\pi \in G} \pi^{-1}(\widetilde{A}^{m})\cap \widetilde{F} \right)>0,
\end{equation*}
and, since $\sigma_1,\ldots, \sigma_n$ are continuous measures, by Remark~\ref{hip}, we can find an element $(z_{i,j})$ in this set, for which moreover
\begin{equation*}
z_{1,j_1}\cdot \ldots \cdot z_{n,j_n}\neq z_{1,j'_1}\cdot \ldots \cdot z_{n,j'_n}
\end{equation*}
whenever $(j_1,\dots,j_n)$ and $(j'_1,\dots,j'_n)$ are distinct elements of $\{1,\dots, m\}^n$.

If $\pi=(\pi_i)_{1\le i\le n}, \pi'=(\pi'_i)_{1\le i\le n} \in G$ are distinct, there exists $1\leq i\leq n$ such that $\pi_i(j)\neq \pi'_i(j)$ for some $1 \leq j \leq m$, whence
\begin{equation*}
(\pi_1(j),\dots \pi_n(j)) \neq (\pi'_1(j),\dots \pi'_n(j))
\end{equation*}
and thus
\begin{equation*}
z_{1,\pi_1(j)}\cdot\ldots \cdot z_{n,\pi_n(j)} \neq z_{1,\pi'_1(j)}\cdot\ldots\cdot z_{n,\pi'_n(j)}.
\end{equation*}
Therefore the elements
\begin{equation*}
(z_{1,\pi_1(1)}\cdot \ldots \cdot z_{n,\pi_n(1)},\ldots,z_{1,\pi_1(m)}\cdot \ldots \cdot z_{n,\pi_n(m)})\in F
\end{equation*}
for $\pi=(\pi_i)_{1\le i\le n}\in G$ are all distinct. Clearly they have the same product of coordinates, and $\# G=(m!)^n$, which completes the proof.
\end{proof}

\begin{wn}\label{wniosek}
Let $\sigma$ be a continuous measure on $\mathbb{T}$. If $V_\sigma^{\odot n}$ has simple spectrum for infinitely many $n\in\mathbb{N}$ then for every $k\geq 1$ the measure $\sigma^{\ast k}$ has the $CS(k+1)$ property.
\end{wn}
\begin{proof}
By Lemma~\ref{vproste}, $V_\sigma^{\odot n}$ has simple spectrum for every $n\in\mathbb{N}$. In particular, when $k\geq 1$ is fixed, $V_\sigma^{\odot mk}$ has simple spectrum for every $m\geq 1$. Let then, for $m\geq 1$,
\begin{equation*}
a_m=(m!)^{k+1}\cdot \frac{(k!)^m}{(mk)!}\,\cdotp
\end{equation*}
By Theorem~\ref{gltw}, it suffices to show that $a_m>1$ for $m$ big enough.
We have
\begin{multline*}
a_{m+1}= a_m \cdot (m+1)^{k+1} \cdot \frac{k!}{(mk+1)\cdot\ldots\cdot((m+1)k)}
\\
\geq a_m \cdot (m+1)^{k+1} \frac{k!}{((m+1)k)^k}=a_m \frac{k!}{k^k}(m+1).
\end{multline*}
Since $k$ is fixed, $a_{m+1}/a_m$ tends to infinity as $m$ increases, which ends the proof.
\end{proof}

\begin{wn}\label{wn:prostygausstocsn}
For every $n \geq 1$ the Cartesian product of $n$ copies of a weakly mixing automorphism whose reduced maximal spectral type generates a Gaussian system with simple spectrum has the CS($n+1$) property.
\end{wn}
\begin{proof}
Let $T\colon (X,\mathcal{B},\mu)\to(X,\mathcal{B},\mu)$ be a weakly mixing automorphism whose reduced maximal spectral type $\sigma$ generates a Gaussian system with simple spectrum. Recall that this implies that $V_\sigma^{\odot n}$ has simple spectrum for all $n$. By Corollary~\ref{wniosek} the measure $\sigma^{\ast k}$ is singular with respect to the convolution of any $k+1$ continuous measures on $\mathbb{T}$ for $k\geq 1$. Therefore
\begin{equation*}
\sigma_{T^{\times n}}=\delta_0+\sum_{k=1}^{n} \sigma^{\ast k} \perp \mu_1 \ast \dots \ast \mu_{n+1}
\end{equation*}
for every continuous measures $\mu_i$ ($1\leq i \leq n+1$) which ends the proof.
\end{proof}

\begin{uw}\label{uwazka}
It is known that a typical automorphism (with respect to the weak operator topology) is weakly mixing (Halmos~\cite{MR0011173}) and its reduced maximal spectral type generates a Gaussian system with simple spectrum (Ageev~\cite{MR1680995}).
\end{uw}

\begin{wn}\label{wn:wnioseczek}
For every $n \geq 1$ the Cartesian product of $n$ copies of a typical automorphism (with respect to the weak operator topology) has the CS($n+1$) property.
\end{wn}
\begin{proof}
It is a direct consequence of Corollary~\ref{wn:prostygausstocsn} and Remark~\ref{uwazka}.
\end{proof}

\subsection{Spectral simplicity of $F_{sym}(U)$}

Let $U\colon H\to H$ be a unitary operator acting on a separable Hilbert space, with continuous maximal spectral type. Let $F_{sym}(U)$ be the corresponding operator on the symmetric Fock space $F_{sym}(H)$. The spectrum of $F_{sym}(U)$ is simple if the following two conditions hold:
\begin{enumerate}
\item
The spectrum of each $U^{\odot k}$ for $k\geq 1$ is simple.
\item
The maximal spectral types of $U^{\odot k}$ are orthogonal: $\sigma_U^{\ast k} \perp \sigma_U^{\ast l}$ for $k\neq l$.
\end{enumerate}
These two conditions are not independent: the first one implies the second one. This follows directly from Corollary~\ref{wniosek}, Remark~\ref{uproste} and Lemma~\ref{vproste}. We also provide below a more precise consequence of the spectral simplicity of symmetric tensor products, with a direct proof.

\begin{pr}\label{pr:bezposc}
Let $\sigma$ be a continuous measure on $\mathbb{T}$. Let $n,m\ge 1$. If $V_\sigma^{\odot (m+n)}$ has simple spectrum then $\sigma^{\ast n}\perp \sigma^{\ast m}\ast \delta_a$
whenever $n\neq m$ or $a\neq 1$.
\end{pr}

\begin{proof}

We may assume that $\sigma$ is a probability measure. Suppose that, for two positive integers $m\le n$ and some $a\in{\mathbb T}$,
\begin{equation}\label{eq:dosz}
\sigma^{\ast n}\not\perp\sigma^{\ast m}\ast \delta_a.
\end{equation}
We will show that, if $n\neq m$ or $a\neq 1$, then $C_{m+n}\colon \mathbb{T}^{m+n}\to \mathbb{T}$ is not one-to-one modulo coordinate permutations on any Borel set $F\subset\mathbb{T}^{m+n}$ with $\sigma^{\otimes (m+n)}(F)=1$, i.e. $V_{\sigma}^{\odot (m+n)}$ does not have simple spectrum.

To avoid further discussion on measurability of direct images, notice that if $A$ is a Borel set in $\mathbb{T}^k$ with $\sigma^{\otimes k}(A)=1$, it contains a $\sigma$-compact set of full measure whose image under $C_k$ is still $\sigma$-compact and carries $\sigma^{\ast k}$. Thus $C_k(A)$ is measurable, with $\sigma^{\ast k}(C_kA)=1$.

Let a Borel set $F\subset \mathbb{T}^{m+n}$ with $\sigma^{\otimes (m+n)}(F)=1$ be given and consider the sets
\begin{equation*}
A_1=\left\{x\in \mathbb{T}^m \colon  \sigma^{\otimes n}(\{y\in\mathbb{T}^n:(x,y)\in F\})=1\right\}
\end{equation*}
and
\begin{equation*}
A_2=\left\{y\in \mathbb{T}^n \colon \sigma^{\otimes m}(\{x\in\mathbb{T}^m:(x,y)\in F\})=1\right\}.
\end{equation*}
By the above observation concerning measurability, $\sigma^{\ast m}(C_mA_1)=\sigma^{\ast n}(C_nA_2)=1$ and thus~\eqref{eq:dosz} implies $C_mA_1\cap a^{-1} C_nA_2 \neq \emptyset$.

Choose $s \in C_mA_1\cap a^{-1}C_nA_2$, $x=(x_1,\dots,x_m)\in A_1$ and $y=(y_1,\dots,y_n)\in A_2$ such that 
$$
C_mx=a^{-1}C_ny=s.
$$
By definition of the sets $A_1$ and $A_2$ there exist Borel sets $B_1\subset\mathbb{T}^m$, $\sigma^{\otimes m}(B_1)=1$ and $B_2\subset\mathbb{T}^n$, $\sigma^{\otimes n}(B_2)=1$ such that 
$$
(x',y)\in F \text{ for all } x'\in B_1
$$
and 
$$
(x,y')\in F\text{ for all }y'\in B_2.
$$
Moreover, since $\sigma$ is continuous, the set of $y'\in \mathbb{T}^n$ with a given coordinate has zero $\sigma^{\otimes n}$-measure and we can restrict $B_2$ to those $y'=(y'_1,\dots,y'_n)$ which have every coordinate $y'_i$ different from every coordinate $y_j$ of $y$.

By~\eqref{eq:dosz} we get again $C_mB_1\cap a^{-1}C_nB_2\ne\emptyset$, so we can choose $s' \in C_mB_1\cap a^{-1}C_nB_2$, $x'=(x'_1,\dots,x'_m)\in B_1$ and $y'=(y'_1,\dots,y'_n)\in B_2$ with $$
C_mx'=a^{-1}C_ny'=s'.
$$
Then $(x,y')\in F,\ (x',y)\in F$, $C_{m+n}(x,y')=C_{m+n}(x',y)=a\cdot s\cdot s'$.

If $m<n$, at least one of the coordinates $y'_i$ does not appear among the $x'_j$ and by the latter assumption on $B_2$ it is not either one of the $y_j$, so $(x,y')$ is not equal to $(x',y)$ modulo a coordinate permutation, which ends the proof in this case.

If $m=n$ and $a\ne1$ then $y'$ cannot be a coordinate permutation of $x'$ since the products of the coordinates are different, and the result follows by the same argument.
\end{proof}

\begin{uw}\label{uw:bezposc}
In the above proof we have not used the fact that we deal with unitary $\mathbb{Z}$-actions. The assertion remains true for continuous unitary representations of any locally compact second countable abelian group.
\end{uw}

\begin{uw}
Let $m>1,\ k\geq 1$. For a continuous measure $\sigma$ on $\mathbb{T}$, the condition that the spectrum of $V^{\odot mk}_{\sigma}$ is simple is essentially stronger than the CS($n$) property for $n\in\mathbb{N}$ such that $(m!)^n>\frac{(mk)!}{(k!)^m}$ (Theorem~\ref{gltw}). Indeed, it suffices to take as $\sigma$ a measure on $\mathbb{T}$ such that the operator $V_{\sigma}^{\odot mk}$ has simple spectrum and consider the representation $V_{\sigma+\sigma\ast \delta_a}$ for some $a\in\mathbb{T}$, $a\neq 1$. The operator $V_{\sigma+\sigma\ast a}^{\odot mk}$ does not have simple spectrum, since this would imply that $V_{\sigma+\sigma\ast \delta_a}^{\odot 2}$ has simple spectrum, while clearly $\left(\sigma+\sigma\ast\delta_a\right)\ast \delta_a \not\perp \sigma+\sigma\ast\delta_a$ which would contradict Proposition~\ref{pr:bezposc}. On the other hand, by Corollary~\ref{wniosek}, $\sigma$ has the CS($n$) property, whence also $\sigma+\sigma\ast \delta_a$ has the CS($n$) property as a sum of two measures which have the property under consideration.

In particular, spectral simplicity of $V_{\sigma}^{\odot 2}$ is a condition which is essentially stronger than the CS property for the measure $\sigma$. 

In view of Remark~\ref{uw:bezposc}, the above discussion remains valid also for other abelian group actions - instead of measures on $\mathbb{T}$ one needs to consider measures on $\hat{\mathbb{G}}$, where $\mathbb{G}$ is the acting group.
\end{uw}

\subsection{Girsanov's theorem}
Proposition~\ref{pr:bezposc} together with methods similar to these used to prove it yield a new proof of the well-known fact that the multiplicity function of a Gaussian system is either identically equal to one or unbounded.
\begin{tw}[\cite{MR0114251}]
Let $U$ be a unitary operator whose (reduced) maximal spectral type $\sigma$ is continuous. Then the set of spectral multiplicities of $\oplus_{n=1}^{\infty}U^{\odot n}$ is either equal to $\{1\}$ or it is unbounded.
\end{tw}
\begin{proof}
Again, we may assume that $\sigma$ is a probability measure. Suppose that the spectrum of $\oplus_{n=1}^{\infty}U^{\odot n}$ is not simple. By Proposition~\ref{pr:bezposc}, there exists $n\geq 1$ such that the spectrum of $U^{\odot n}$ is not simple. Assume that for some $n\ge 1$ the spectral multiplicity of $U^{\odot n}$ is at least equal to some integer $q\ge 2$. This means that for every measurable set $F\subset \mathbb{T}^n$ of full $\sigma^{\otimes n}$-measure we can find $s\in\mathbb{T}$ and $q$ points
\begin{equation*}
{x}_i=(x_i^1,\dots,x_i^n)\in F\ \ (1\leq i\leq q)
\end{equation*}
so that $x_i^1\cdot\ldots\cdot x_i^n=s$ for $1\leq i\leq q$ but none of the points ${x}_i$ can be obtained from another one in this set by coordinate permutation.

Let $E$ be any Borel set in $\mathbb{T}^{2n}$ with $\sigma^{\otimes 2n}(E)=1$ and
\begin{equation*}
A=\left\{{x}\in\mathbb{T}^n\colon \sigma^{\otimes n}\{{y}\in\mathbb{T}^n\colon ({x},{y})\in E\}=1\right\}.
\end{equation*}
We have
\begin{equation*}
\sigma^{\otimes n}(A)=1.
\end{equation*}

Let $s\in\mathbb{T}$ and ${x}_i=(x_i^1,\dots,x_i^n)\in A$ $(1\le i\le q)$ be such that $x_i^1\cdot\ldots\cdot x_i^n=s$ for $1\leq i\leq q$ and none of the points ${x}_i$ can be obtained from another one in this set by coordinate permutation. Let
\begin{equation*}
B=\left\{{y}\in\mathbb{T}^n \colon ({x}_i,{y})\in E \text{ for all }1\leq i\leq q\right\}.
\end{equation*}
Notice that $\sigma^{\otimes n}(B)=1$ and, since $\sigma$ is continuous, we can assume without loss of generality that the coordinates of points in $B$ are different from all of the coordinates of the points ${x}_i$ for $1\leq i\leq q$. 

Let $s'\in\mathbb{T}$ and ${y}_j\in B$ $(1\leq j\leq q)$ be such that $y_j^1\cdot\ldots\cdot y_j^n=s'$ for $1\leq j\leq q$ and none of the points ${y}_j$ can be obtained from another one in this set by coordinate permutation. Then
\begin{equation*}
({x}_i,{y}_j)\in E
\end{equation*}
and
\begin{equation*}
x_i^1\cdot\ldots\cdot x_i^n\cdot y_j^1\cdot\ldots\cdot y_j^n=s\cdot s'
\end{equation*}
for $1\leq i,j\leq q$. Moreover none of the points $({x}_i,{y}_j)$ for $1\leq i,j\leq q$ can be obtained from another one from this set by a permutation of coordinates. Therefore the spectral multiplicity of the operator $U^{\odot 2n}$ is at least $q^2$. It follows inductively that the spectral multiplicity function of $\oplus_{n=1}^{\infty}U^{\odot n}$ is unbounded, which completes the proof.
\end{proof}

\section{CS($n$) implies JP($n-1$)}\label{sekcja}
Let $n\geq 1$, let $T$ be an ergodic automorphism acting on $(X,\mathcal{B},\mu)$ and $S_i$ be weakly mixing automorphisms acting respectively on $(Y_i,\mathcal{C}_i,\nu_i)$ for $1\leq i \leq n$. Notice that 
\begin{equation}\label{eq:lm:20:1-pr}
L^2_0(Y_1\times \dots \times Y_n)= \bigoplus_{1\leq k\leq n}\, \bigoplus_{1\leq i_1<\dots <i_k\leq n}L_{i_1,\dots,i_k}^Y,
\end{equation}
where for $1\leq i_1<\dots <i_k\leq n$
\begin{equation*}
L_{i_1,\dots,i_k}^Y=\otimes_{j=1}^k L^2_0(Y_{i_j}).
\footnote{For $k\geq 1$ we treat the elements of $L^Y_{i_1,\dots,i_k}$ as functions of $n$ variables.\label{traktujemy}}
\end{equation*}
For $1\leq i_1<\dots<i_k\leq n$ let
$$p_{i_1,\dots,i_k}\colon L^2(Y_1\times\dots\times Y_n)\to L^2(Y_{i_1}\times\dots\times Y_{i_k})$$
and
$$p_0\colon L^2(Y_1\times\dots\times Y_n)\to \mathbb{C}$$ 
stand for the orthogonal projections.

Recall the following well-known fact.
\begin{lm}\label{lm:kk}
Given $T\colon (X,\mathcal{B},\mu)\to(X,\mathcal{B},\mu)$, $S\colon (Y,\mathcal{C},\nu)\to (Y,\mathcal{C},\nu)$, a factor of $\mathcal{A}\subset \mathcal{C}$ of $S$ and $\lambda\in J(T)$, $proj_{L^2(\mathcal{A})}\circ \Phi_\lambda$ is the Markov operator corresponding to the relatively independent extension of $\lambda|_{\mathcal{B}\otimes \mathcal{A}}$.
\end{lm}
\begin{proof}
Let $B\in \mathcal{B}, C\in\mathcal{C}$. Since the orthogonal projection is a self-conjugate operator, looking at the scalar product we have
\begin{multline*}
\int_Y proj_{L^2(\mathcal{A})} \circ \Phi_\lambda (\1_{B})(x)\1_{C}(x)\ d\mu\\
=\int_Y \Phi_\lambda (\1_{B})(x)\ proj_{L^2(\mathcal{A})}(\1_{C})(x)\ d\mu(x)\\
= \int_{X\times Y}\1_{B}(x)\ E(\1_C|\mathcal{A})(y)\ d\lambda (x,y).
\end{multline*}
\end{proof}
We will use the above lemma in the situation where $S$ is a direct product of its factor $S|_\mathcal{A}$ with some other transformation.

\begin{lm}\label{lm:lemat}
Let $\lambda \in J^e(T,S_1\times \dots \times S_n)$. If
\begin{equation*}
\Phi_{\lambda}(L^2(X,\mathcal{B},\mu)) \perp L^Y_{1,\dots,n}
\end{equation*}
then there exists $1\leq i\leq n$ such that
\begin{equation*}
\lambda=\lambda_{X,Y_1,\dots,Y_{i-1},Y_{i+1},\dots,Y_n}\otimes \nu_i.
\end{equation*}
\end{lm}
\begin{proof}
Let $p=proj_{L^Y_{1,\dots,n}}$. Then
$$
p=\otimes_{i=1}^{n}(Id - q_i),
$$
where $q_i\colon L^2(Y_i)\to \mathbb{C}$, $1\leq i\leq n$ stands for the orthogonal projection. Notice that for $1\leq i_1<\dots<i_l\leq n$
$$
q_{i_1}\otimes \dots \otimes q_{i_l}=p_{j_1,\dots,j_{n-l}},
$$
where $\{j_1,\dots,j_{n-l}\}=\{1,\dots,n\}\setminus \{i_1,\dots,i_l\}$. We have
\begin{multline*}
proj_{(L^Y_{1,\dots,n})^\perp} = Id - p = Id - \otimes_{i=1}^{n}(Id-q_i)\\
= \sum_{l=1}^{n}(-1)^{l-1}\sum_{1\leq i_1<\dots i_l\leq n}q_{i_1}\otimes \dots\otimes q_{i_l}\\
=(-1)^{n-1}p_0+ \sum_{k=0}^{n-1} (-1)^{n-k-1}\!\!\!\sum_{1\leq i_1<\dots<i_k\leq n}p_{i_1,\dots,i_k}.
\end{multline*}

By the assumption that $\Phi_{\lambda}(L^2(X))\perp L^Y_{1,\dots,n}$,
it follows that
\begin{equation*}
\Phi_{\lambda}=(-1)^{n-1}p_0 \circ\Phi_{\lambda}+\sum_{k=0}^{n-1} (-1)^{n-k-1}\!\!\!
\sum_{1\leq i_1<\dots<i_k\leq n} p_{i_1,\dots,i_k}\circ\Phi_{\lambda}.
\end{equation*}
By Lemma~\ref{lm:kk} the Markov operator $p_{i_1,\dots,i_k} \circ \Phi_{\lambda}$ corresponds to the relatively independent extension of $\lambda_{X,Y_{i_1},\dots,Y_{i_k}}$, i.e. $\lambda_{X, Y_{i_1},\dots, Y_{i_k}} \otimes \nu_{j'_1}\otimes \dots \otimes \nu_{j'_{n-k}}$ where $\{j'_1,\dots,j'_{n-k}\}=\{1,\dots,n\}\setminus\{i_1,\dots,i_k\}$.
Hence
\begin{multline}\label{eq:mmm}
\lambda=(-1)^{n-1} \nu_1\otimes \dots \otimes \nu_n\\
+\sum_{k=0}^{n-1} (-1)^{n-k-1}\!\!\!\sum_{1\leq i_1<\dots<i_k\leq n} \lambda_{X,Y_{i_1},\dots, Y_{i_k}} \otimes \nu_{j'_1}\otimes \dots \otimes \nu_{j'_{n-k}}.
\end{multline}
All the joinings appearing in the above expression are ergodic:
\begin{itemize}
\item
$\lambda$ - as a member of $J^e(R,S_1\times \dots \times S_n)$ by assumption,
\item
$\lambda_{X,Y_{i_1},\dots, Y_{i_k}}$ - as a projection of the ergodic measure $\lambda$,
\item
$\lambda_{X,Y_{i_1},\dots, Y_{i_k}} \otimes \nu_{j_1}\otimes \dots \otimes \nu_{j_{n-k}}$ - by the assumption that the $S_i$ are weakly mixing.
\end{itemize}
Now we write~\eqref{eq:mmm} as an equality between sums of ergodic joinings. By the uniqueness of the ergodic decomposition, $\lambda$ is equal to one of the other measures which completes the proof.
\end{proof}

\begin{tw}\label{tw:4}
Let $T$ be an ergodic automorphism of a standard probability Borel space $(X,\mathcal{B},\mu)$. Whenever it enjoys the CS($n$) property for some $n\geq 2$, it also enjoys the JP($n-1$) property.
\end{tw}

\begin{proof}
Let $T \colon (X,\mathcal{B},\mu)\to (X,\mathcal{B},\mu)$ satisfy the assumptions of the theorem. Let $n\geq 2$ and let $\lambda \in J^e(T,S_1\times \dots \times S_n)$ for some weakly mixing automorphisms $S_1,\dots,S_n$. Recall that for any $f\in L^2(X,\mu)$ we have
\begin{equation*}
\sigma_{\Phi_{\lambda}(f)} \ll \sigma_f \ll \sigma_T
\end{equation*}
(see e.g.~\cite{MR1784644}) and that the maximal spectral type of $S_1\times\dots \times S_n$ on $L^Y_{1,\dots,n}$ is the convolution of the maximal spectral types of the automorphisms $S_1,\dots,S_n$ on $L^2_0(Y_1),\dots,L^2_0(Y_n)$ respectively. Using the decomposition~\eqref{eq:lm:20:1-pr} of $L^2_0(Y_1\times \dots \times Y_n)$ and the assumption that $\sigma_T \perp \mu_1 \ast \dots \ast \mu_n$ for any continuous measures $\mu_1,\dots,\mu_n$ we obtain
\begin{equation*}
\Phi_{\lambda}(L^2(X,\mathcal{B},\mu)) \perp L^Y_{1,\dots,n}.
\end{equation*}
Therefore the assumption of Lemma~\ref{lm:lemat} is fulfilled, which completes the proof.
\end{proof}

%
%
%


\section{JP($n-1$)$\subsetneq$ JP($n$) for $n\geq 2$}\label{se:5}
In this section we will show that $\textrm{JP(}n-1\textrm{)}\neq \textrm{JP(}n\textrm{)}$ for $n\geq 2$ by giving examples of automorphisms which are in JP($n$) but not in JP($n-1$). 

\begin{lm}\label{lm:lemac}
Let $T\colon (X,\mathcal{B},\mu)\to(X,\mathcal{B},\mu)$ be a weakly mixing automorphism. Then $T^{\times n}\notin \textrm{JP(}n-1\textrm{)}$ for $n\geq 2$.
\end{lm}
\begin{proof}
Let $n\geq 2$. To see that the assertion is true, it suffices to consider the diagonal joining $\Delta$ of  $T^{\times n}$ with itself.\footnote{Given an automorphism $S\colon (Y,\mathcal{C},\nu)\to (Y,\mathcal{C},\nu)$ the diagonal self-joining $\Delta$ is defined by the formula $\Delta(A\times B)=\nu(A\cap B)$.}

\end{proof}
As a direct consequence of Corollary~\ref{wn:prostygausstocsn} and Theorem~\ref{tw:4} we obtain the following.
\begin{wn}\label{wni}
For any $n\geq 1$ and any weakly mixing automorphism $T$ whose reduced maximal spectral type $\sigma$ generates a Gaussian system with simple spectrum, the Cartesian product $T^{\times n}$ of $n$ copies of $T$ has the JP($n$) property. In particular, for any $n\geq 1$ the Cartesian product of $n$ copies of a typical automorphism (with respect to the weak operator topology) has the JP($n$) property.
\end{wn}
Moreover, using Corollary~\ref{wni} and Lemma~\ref{lm:lemac} we have:
\begin{wn}\label{dif}
For $n\geq 2$ the class of automorphisms enjoying the JP($n-1$) property is a proper subclass of automorphisms enjoying the JP($n$) property.
\end{wn}

\begin{uw}\label{uw:stary}
In~\cite{MR2729082} there is an example of a JP($2$) system which is not in the class JP($1$). The methods used there are different than what we use to obtain Corollary~\ref{dif}. It would be interesting to know if it is possible to find examples of systems from the class JP($n$)$\setminus$JP($n-1$) not using the methods from the present paper.
\end{uw}

\begin{uw}
Recall~\cite{MR2729082} that the class JP is closed under taking distal extensions which are weakly mixing. The proof from~\cite{MR2729082} can be rewriten almost word by word to obtain that 
\begin{enumerate}
\item[(i)]
each JP($n$) class is closed under taking distal extensions which are weakly mixing.
\end{enumerate}
Moreover,
\begin{enumerate}
\item[(ii)]
each JP($n$) class is closed under taking factors and inverse limits.
\end{enumerate}
Moreover (see \cite{MR2729082}),
\begin{enumerate}
\item[(iii)]
every system which is distally simple\footnote{The notion of distal simplicity was defined in~\cite{delJ-Lem}. It imposes a restriction on the self-joinings of the considered system and is a generalization of the notion of quasi-simplicity~\cite{MR2265691}.} has the JP($n$) property,
\item[(iv)]
the class of distally simple systems is closed under factors, distal extensions and inverse limits.
\end{enumerate}
In~\cite{MR2558487} a necessary condition for lifting the JP property for the so-called Rokhlin extensions is provided.\footnote{Recall that given a locally compact group $G$, a cocycle $\varphi\colon X\to G$ and a measurable $G$-action $\mathcal{S}=(S_g)_{g\in G}$ on $(Y,\mathcal{C},\nu)$ the relevant Rokhlin extension of $T\colon (X,\mathcal{B},\mu)\to (X,\mathcal{B},\mu)$ is given by $T_{\varphi,\mathcal{S}}(x,y)=(Tx,S_{\varphi(x)}(y))$.} This yields a class of automorphisms with the JP property, not having the DS property.

We do not know if the examples from Remark~\ref{uw:stary}, from~\cite{MR2558487}, from Corollary~\ref{dif} and (i)-(iv) are the only ``ways'' to obtain systems with the JP($n$) property.

\end{uw}

\section{JP($n$) property and disjointness with ID systems}

In~\cite{MR2729082} Lema\'{n}czyk, Parreau and Roy showed that 
\begin{equation}\label{oni}
\begin{array}{l}
\mbox{all systems with the JP($n$) property are disjoint }\\
\mbox{from ergodically infinitely divisible automorphisms.}
\end{array}
\end{equation}
Let us recall that an ergodic automorphism $T\colon (X,\mathcal{B},\mu)\to (X,\mathcal{B},\mu)$ is said to be infinitely divisible if there exists a sequence of factors $\{\mathcal{B}_{\omega}\colon \omega\in \{0,1\}^\ast\}$\footnote{$\{0,1\}^{\ast}$ stands for the set of all finite sequences with entries $0$ and $1$.} of $\mathcal{B}$ where $\mathcal{B_{\varepsilon}}=\mathcal{B}$, $\mathcal{B}_{\omega}=\mathcal{B}_{\omega 0}\otimes \mathcal{B}_{\omega 1}$ and for each $f\in L^2_0(X,\mathcal{B},\mu)$ and $\eta\in \{0,1\}^{\mathbb{N}}$ it holds
$$
\lim_{n\to\infty} E(f|\mathcal{B}_{\eta[0,n)})=0.
$$
In this section we will deal with another notion of infinite divisibility, namely we will consider dynamical systems arising from stationary infinitely divisible processes. Let us recall the definition of such processes.
\begin{df}[see e.g.~\cite{MR2729082}]
An ergodic stationary process $(X_n)_{n\in\mathbb{Z}}$ on a standard probability Borel space $(X,\mathcal{B},\mu)$ is \emph{infinitely divisible} (ID) if its distribution $P$ on $(\mathbb{R}^{\mathbb{Z}},\mathcal{B}^{\otimes \mathbb{Z}})$ is such that for any $k\geq 1$ there exists a probability measure $P_k$ on $(\mathbb{R}^{\mathbb{Z}},\mathcal{B}^{\otimes \mathbb{Z}})$ such that $P=P_k^{\ast k}$\footnote{Recall that $P_k^{\ast k}$ stands for the $k$-th convolution power of $P_k$.}.
\end{df}
\begin{uw}[see e.g.~\cite{MR2729082}]
A stationary process $(X_n)_{n\in\mathbb{Z}}$ is infinitely divisible if and only if for all $k\geq 1$ and $n\in\mathbb{N}$ there exist pairwise independent random variables $X_n^{(1,k)}, \dots, X_n^{(k,k)}$ such that $X_n=X_n^{(1,k)}+\dots+X_n^{(k,k)}$ and that processes $\left(X_n^{(i,k)}\right)_{n\in\mathbb{Z}}$, $1\leq i\leq k$, are stationary with same distribution $P_k$.
\end{uw}

The following proposition describes the relation between systems arising from ID stationary processes and ergodically ID automorphisms.
\begin{pr}[\cite{MR2729082}]
ID stationary processes are factors of ergodically ID dynamical systems.
\end{pr}
An immediate consequence of~\eqref{oni} and of the above proposition is the following corollary.
\begin{wn}[\cite{MR2729082}]
All systems with the JP($n$) property are disjoint from systems arising from ID stationary processes.
\end{wn}

We will provide another proof of this result which will be based on a more general proposition. Let us first recall some definitions and introduce necessary notation.
\begin{df}
Let $T$ be an automorphism on a standard probability Borel space $(X,\mathcal{B},\mu)$ and let $k\geq 1$. The sub-$\sigma$-algebra of $\mathcal{B}^{\otimes k}$ which consists of all sets invariant under permutations of coordinates is called the \emph{symmetric factor} of $T^{\times k}$. It is denoted by $\mathcal{F}_k(T)$.
\end{df}

\begin{uw}[\cite{MR2729082}]
The dynamical system determined by a stationary ID process is, for any integer $k\geq 1$, a factor of the symmetric factor of the Cartesian $k$-th power of a dynamical system.
\end{uw}

\begin{pr}\footnote{Lema\'{n}czyk, Parreau and Roy in~\cite{MR2729082} (Proposition 5) cover the case $n=1$.}\label{to}
Fix $n\geq 1$ and let $T$ be an automorphism of a standard probability space $(X,\mathcal{B},\mu)$. Assume that, for each $g\in L^2(X,\mathcal{B},\mu)$, there exists a sequence $(k_j)_{j\geq 1}$ of integers going to infinity and a sequence $(S_j)_{j\geq 1}$ of weakly mixing automorphisms such that $T$ is a factor of $S_{j}^{\times k_j}$ and moreover
\begin{equation*}
dist(g,L^2(\mathcal{F}_{k_j}(S_j)))=o\left(\frac{1}{k_j^{\nicefrac{n}{2}}} \right).
\end{equation*}
Then $T$ is disjoint from every JP($n$) automorphism.
\end{pr}
\begin{proof}
Let $\lambda_0\in J^e(R,T)$, where $R:\zdr\to\zdr$ is an ergodic automorphism with the JP($n$) property
and let $f\in L^2_0\zdr$.
It is enough to show that $g:=\Phi_{\lambda_0}(f)=0$.

By the assumption, given any $\vep>0$, there exist an
arbitrarily large
integer $k\geq n$ and a weakly mixing automorphism $S$ on a standard
space $\ycn$ such that $T$ is a factor of $S^{\times k}$ and
\begin{equation}\label{f0}
dist(g,L^2(\mathcal{F}_{k}(S))\le\frac{\vep}{k^{\nicefrac{n}{2}}}\;\cdotp
\end{equation}
Let $\hat{\lambda}_0$ be the relatively independent extension
of $\lambda_0$ to a
joining of $R$ with $S^{\times k}$, so that
$\Phi_{\hat{\lambda}_0}$ is the composition
of $\Phi_{\lambda_0}$ and the embedding of $L^2\xbm$ into
$L^2(Y^{\times k},\mathcal{C}^{\otimes k},\nu^{\otimes k})$.  For the sake of simplicity of notation we will assume that $\mathcal{B}\subset \mathcal{C}^{\otimes k}$ and $\mu = \nu^{\otimes k}|_{\mathcal{B}}$. Then we have $\Phi_{\hat{\lambda}_0}(f)=\Phi_{\lambda_0}(f)=g$. 

Consider now the ergodic decomposition of $\hat{\lambda}_0$:
\begin{equation*}
\hat{\lambda}_0=\int_{J^e(R,S^{\times k})} \lambda\ dP(\lambda).
\end{equation*}
Since $R$ has the JP($n$) property, for each $\lambda\in J^e(R,S^{\times k})$ there exist $1\leq i_{\lambda,1}<\dots <i_{\lambda,n}\leq k$ such that 
$$
\Phi_{\lambda}(f) \in L^2_0(Y_{i_{\lambda,1}}\times \dots \times Y_{i_{\lambda,n}}).
$$
Then (see~\eqref{eq:lm:20:1-pr})
\begin{equation}\label{eq:nowe8}
\Phi_{\lambda}(f)=\sum_{m=1}^n\sum_{\substack{ { j_1<\dots<j_m}\\  \{j_1,\dots,j_m\}\subset\{i_{\lambda,1},\dots,i_{\lambda,n}\}}}p'_{j_1,\dots,j_m}\circ\Phi_{\lambda}(f),
\end{equation}
where $p'_{j_1,\dots,j_m}$ stands for the orthogonal projection from $L^2(Y_1\times\dots\times Y_n)$ onto $L^Y_{j_1,\dots,j_m}$. For $1\leq j_1<\dots<j_m\leq k$ we also have
\begin{equation}\label{eq:par13}
{p'_{j_1,\dots,j_m}}\circ\Phi_{\lambda}(f)=0 \text{ whenever }\{j_1,\dots,j_m\}\not\subset \{{i}_{\lambda,1},\dots,{i}_{\lambda,n}\}.
\end{equation}
Hence, letting for $1\leq m\leq n$
\begin{equation*}
f_{j_1,\dots,j_m}=\int_{J^e(R,S^{\times k})}p'_{j_1,\dots,j_m}\circ \Phi_{\lambda}(f)\ dP(\lambda)
\end{equation*}
for $1\leq j_1<\dots<j_m\leq k$, and
$$
g_m=\sum_{1\leq j_1<\dots<j_m\leq k}f_{j_1,\dots,j_m} \in \bigoplus_{1\leq j_1<\dots<j_m\leq k}L_{j_1,\dots,j_m}^Y
$$
we obtain by~\eqref{eq:nowe8} and~\eqref{eq:par13} that
$$
g=\sum_{m=1}^n g_m.
$$

Fix $1\leq m \leq n$. By~\eqref{eq:par13} for each $\lambda\in J^e(R,S^{\times k})$ at most ${n \choose m}$ out of the ${k \choose m}$ projections $p'_{s_1,\dots,s_m}\circ\Phi_{\lambda}(f)$ do not vanish. Therefore there exist $j_1^0<\dots<j_m^0$ such that
\begin{equation*}
P\left(\left\{\lambda\colon p'_{j_1^0,\dots,j_m^0}\circ\Phi_{\lambda}f\neq 0\right\}\right)\leq\frac{{n \choose m}}{{k\choose m}},
\end{equation*}
whence
\begin{equation}\label{f2}
\|f_{j_1^0,\dots,j_m^0}\|\leq \frac{{n \choose m}}{{k\choose m}}\|f\|.
\end{equation}
Now we claim that for all $j_1<\dots<j_m$ the following inequality holds:
\begin{equation}\label{f1}
\left| \| f_{j_1,\dots,j_m}\|-\|f_{j_1^0,\dots ,j_m^0} \|\right|\leq \frac{2\varepsilon}{k^{\nicefrac{n}{2}}}.
\end{equation}
Indeed, consider a permutation $\pi$ of $\{1,\dots k\}$ sending each $j_s^0$ to $j_s$ for all $s\in \{1,\dots,m\}$ and the corresponding unitary operator $U_{\pi}$ on $L^2(Y^{\times k},\mathcal{C}^{\otimes k},\nu^{\otimes k})$. Since all functions in $L^2(\mathcal{F}_k(S))$ are fixed by $U_{\pi}$,
\begin{multline*}
\|f_{j_1,\dots,j_m}-f_{j_1^0,\dots,j_m^0}\|\\
\leq \left(\sum_{1\leq l \leq n}\sum_{1\leq i_1<\dots< i_l\leq k}\|f_{\pi(i_1),\dots,\pi(i_l)}-f_{i_1,\dots,i_l} \|^2\right)^{\nicefrac{1}{2}}\\
=\|U_{\pi}g-g\|\leq 2\cdot\text{dist}(g,L^2(\mathcal{F}_k(S)))
\end{multline*}
and~\eqref{f1} follows from~\eqref{f0}.

So, by~\eqref{f2}, $\|f_{j_1,\dots j_m}\|\leq \frac{{n \choose m}}{{k\choose m}}\|f\|+\frac{2\varepsilon}{k^{\nicefrac{n}{2}}}$. Therefore 
\begin{multline*}
\|g_m\| =\left( \sum_{1\leq j_1<\dots<j_m\leq n} \|f_{j_1,\dots,j_m}\|^2\right)^{\nicefrac{1}{2}}\\
\leq {k \choose m}^{\nicefrac{1}{2}}\left(\frac{{n\choose m}}{{k\choose m}}\|f\|+\frac{2\vep}{k^{\nicefrac{n}{2}}} \right)
\leq \frac{{n \choose m}}{{k\choose m}^{\nicefrac{1}{2}}}\|f\|+2\vep.
\end{multline*}
Since $k$ was arbitrarily large and $\varepsilon$ arbitrarily small, this proves that $g_m=0$. Hence $g=\sum_{m=1}^{n}g_m=0$.
\end{proof}

\begin{wn}
Let $n\geq 1$. Any non-zero root of an automorphism with the JP($n$) property is disjoint from all automorphisms arising from ID stationary processes.
\end{wn}
\begin{proof}
It suffices to notice that whenever the assumptions of Proposition~\ref{to} are satisfied for some automorphism, then they are also satisfied for any non-zero power of it. This implies the disjointness of roots of automorphisms with the JP($n$) property from these automorphisms $T$, in particular the disjointness from automorphisms arising from stationary ID processes.
\end{proof}

One of the consequences of the fact that JP systems are disjoint from systems coming from ID stationary processes is the following example of a system which has the so-called Kolmogorov group property, i.e. the convolution of two copies of its maximal spectral type is absolutely continuous with respect to its maximal spectral type. Let $T$ acting on $(X,\mathcal{B},\mu)$ be an automorphism whose reduced maximal spectral type generates a Gaussian system with simple spectrum, and let $T_i \colon (X_i,\mathcal{B}_i,\mu_i)\to (X_i,\mathcal{B}_i,\mu_i)$ for $i\geq 1$ be isomorphic copies of it. Consider the infinite Cartesian product $R=T_1\times T_2\times \dots$. Denote by $\sigma$ the maximal spectral type of $T$. Then $\exp(\sigma):=\sum_{n=1}^{\infty}\frac{\sigma^{\ast n}}{n!}$ is the maximal spectral type of $R$ and clearly  $\sigma_R\ast \sigma_R \ll\sigma_R$. We claim that $R$ is disjoint from automorphisms arising from stationary ID processes. Let $S\colon (Y,\mathcal{C},\nu)\to  (Y,\mathcal{C},\nu)$ be such an automorphism and consider an ergodic joining  $\lambda\in J^e(R,S)$. Then for any $n\in\mathbb{N}$
$$
\lambda|_{X_1,\dots,X_n,Y} \in J(T_1\times \dots \times T_n, S).
$$
By Corollary~\ref{wni} and Proposition~\ref{to}
$$
\lambda|_{X_1,\dots,X_n,Y}=\mu_1\otimes \dots \otimes \mu_n\otimes \nu.
$$
This implies that $\lambda=(\mu_1\otimes \mu_2\otimes \dots) \otimes \nu$.

\section*{Acknowledgements}
We would like to thank M.~Lema{\'n}czyk and E.~Roy for fruitful discussions.

\footnotesize
\bibliography{fok}

\end{document}